\documentclass[11pt]{article}
\usepackage{anysize}\marginsize{3.5cm}{3.5cm}{1.3cm}{2cm}

\makeatletter\@ifundefined{pdfpagewidth}{}{\pdfpagewidth=21.0cm\pdfpageheight=29.7cm}\makeatother 

\usepackage{amssymb}
\usepackage[latin1]{inputenc}  
\headsep=\baselineskip
\makeatletter

\let\orig@item=\@item \def\@item[#1]{\orig@item[\rm #1]}
\renewenvironment{abstract}{\begin{quote}\footnotesize\textbf{\abstractname.}}{\end{quote}\bigskip}
\renewcommand\l@section{\@dottedtocline{1}{0em}{1.6em}} 
\renewcommand\l@subsection{\@dottedtocline{2}{1.6em}{2em}} 
\renewcommand\@seccntformat[1]{\csname the#1\endcsname.\enspace}
\renewcommand\paragraph{\@startsection{paragraph}{4}{\z@}{1\baselineskip}{-0.5em}{\normalsize\bfseries}}
\let\origcaption=\caption \renewcommand\caption[1]{\parbox{0.66\textwidth}{\origcaption{#1}}}

\renewcommand\@begintheorem[2]{\trivlist\item[\hskip\labelsep{\bfseries#1 #2.}]\it}
\renewcommand\@opargbegintheorem[3]{\trivlist\item[\hskip\labelsep{\bfseries#1 #2}] {\bfseries(#3).}\enspace\it\ignorespaces}
\makeatother

\newtheorem{satz}{Satz}[section]
\makeatletter\@addtoreset{equation}{satz}\makeatother

\newtheorem{theorem}[satz]{Theorem}

\newtheorem{proposition}[satz]{Proposition}
\newtheorem{remark}[satz]{Remark}
\newtheorem{algorithm}[satz]{Algorithm}

\newenvironment{proclaim}[1]{\begin{list}{}{\labelwidth=0cm\leftmargin=0cm}\item[\hspace{\labelsep}\bfseries#1.]\itshape}{\end{list}}
\newenvironment{proof}[1][Proof]{\trivlist\item[\hskip\labelsep{\it #1.}]}{\hspace*{\fill}$\Box$\endtrivlist}

\newcommand\compact{\itemsep=0cm \parskip=0cm}

\newcommand\engqq[1]{``#1''}
\newcommand\tab[2][t]{\begin{tabular}[#1]{@{}l@{}}#2\end{tabular}}

\renewcommand\emptyset{\varnothing}  
\renewcommand\ge{\geqslant}  
\renewcommand\le{\leqslant}  
\renewcommand\epsilon{\varepsilon}
\renewcommand\phi{\varphi}

\renewcommand\O{{\cal O}}
\renewcommand\P{\mathbb P}

\newcommand\liste[3]{\mbox{$#1_{#2},\dots,#1_{#3}$}}
\newcommand\longsum[3]{#1_{#2}+\dots+#1_{#3}} 
\newcommand\set[1]{\left\{#1\right\}}
\newcommand\sset[1]{\left\{\,#1\,\right\}}
\newcommand\eqnref[1]{(\ref{#1})}
\newcommand\eqdef{\stackrel{\rm def}{=}}
\newcommand\vect[1]{\left(\begin{array}{c}#1\end{array}\right)}
\newcommand\union{\cup}
\newcommand\intersect{\cap}
\newcommand\maxmatrcols{10}
\newlength\matrcolsep \matrcolsep=\arraycolsep
\newcommand\matr[1]{{\arraycolsep=\matrcolsep\left(\begin{array}{*{\maxmatrcols}{c}}#1\end{array}\right)}}
\newcommand\inverse{^{-1}}
\newcommand\tabline{\rule[-1.1ex]{0pt}{1ex}\\\hline\rule{0pt}{2.8ex}}

\newcommand\N{\mathbb N}
\newcommand\R{\mathbb R}
\newcommand\Z{\mathbb Z}

\newcommand\newop[2]{\newcommand#1{\mathop{\rm #2}\nolimits}}
\newop\NS{NS} 
\newop\vol{vol}
\newop\BigCone{Big} 
\newop\B{Bs}
\newop\SB{\textbf{B}}

\newcommand\dd[2]{\delta_{#1#2}}
\newcommand\CC[1]{C^{(#1)}}

\newenvironment{lines}
   {\newcommand\+{\hspace*{2em}\ignorespaces} 
   \begin{list}{}{\leftmargin=\leftmargini\parsep=0cm}\item\begin{obeylines}}
   {\end{obeylines}\end{list}}

\begin{document}

\title{Counting Zariski chambers on Del Pezzo surfaces}
\author{Thomas Bauer, Michael Funke, Sebastian Neumann}
\date{December 4, 2009}
\maketitle
\thispagestyle{empty}

\begin{abstract}
   Zariski chambers provide a natural decomposition of the big cone
   of an algebraic surface into
   rational locally polyhedral subcones that are
   interesting from the point of view of linear series.
   In the present paper we present an
   algorithm that allows to effectively determine Zariski
   chambers when the negative curves on the surface are known.
   We show how the algorithm can be used to compute the number of
   chambers on Del Pezzo surfaces.
\end{abstract}

\section*{Introduction}

   In \cite{BKS} it was shown that the big cone of an algebraic
   surface admits a natural
   locally finite decomposition into
   rational locally polyhedral
   subcones, the \emph{Zariski chambers} on $X$.
   These chambers are of basic interest from the point of view of
   linear series on $X$: In the interior of each Zariski chamber the
   stable base loci are constant, and the volume function is
   given by a
   quadratic polynomial in each chamber.
   (See Sect.~\ref{sect-chambers} for details on the chamber
   decomposition.)
   Understanding the behaviour of stable base loci and the
   volume function is also of great interest
   in the higher-dimensional case, where the picture is not
   as clear as for surfaces
   (see \cite{AILB} and \cite{AIBL}).

   It is an intruiging question to wonder into how many Zariski
   chambers the big cone decomposes on a given surface.
   In other words, we ask
   on a smooth projective surface $X$ for
   the quantity
   $$
      z(X)=\#\set{\mbox{Zariski chambers on $X$}}\in\N\union\set\infty \,.
   $$
   The number $z(X)$ is an interesting geometric invariant of the
   surface $X$, as it is the answer to the following questions
   (see Sect.~1):
   \begin{itemize}\compact
   \item
      How many different stable base loci can occur in big linear
      series on $X$ ?
  \item
      How many essentially different Zariski decompositions can
      big divisors on $X$ have? (By
      \engqq{essentially different} we mean here that their negative
      parts have different support.)
   \item
      How many \engqq{pieces} does the volume function
      $\vol\colon\BigCone(X)\to\R$ have
      (which is a piecewise polynomial function)?
   \end{itemize}
   So, somewhat roughly speaking, one may think of
   the number $z(X)$ as
   measuring how complicated the surface is from the point of
   view of linear series.

   In the present paper we provide an algorithm that allows to
   compute the invariant
   $z(X)$ whenever the irreducible curves of negative
   self-intersection on $X$ are known.
   In particular, we will show how to apply the algorithm to
   Del Pezzo surfaces.
   Recall that a
   Del Pezzo surface is either $\P^1\times\P^1$, $\P^2$, or
   a blow-up of $\P^2$ at $r\le 8$
   general points.
   As one clearly has $z(\P^1\times\P^1)=1$ and $z(\P^2)=1$,
   it is enough to study the blow-ups.
   We show:

\begin{proclaim}{Theorem}
   Let $X_r$ be the blow-up of $\P^2$ in $r$ general points with
   $1\le r\le 8$.
   \begin{itemize}\compact
   \item[(i)]
   The number $z(X_r)$ of Zariski chambers on $X_r$ is given by
   the following table:
   $$
      \begin{array}{c|*8c}
         r        & 1 & 2 & 3 & 4 & 5 & 6 & 7 & 8 \tabline
         z(X_r)   & 2 & 5 & 18 & 76 & 393 & 2\,764 & 33\,645 & 1\,501\,681
      \end{array}
   $$

   \item[(ii)]
   The maximal number of curves that occur in the support of a
   Zariski chamber on $X_r$ is $r$.
   \end{itemize}
\end{proclaim}

   As one might expect intuitively,
   the number of chambers increases as
   the Picard number $\rho(X_r)=r+1$ increases.
   Note however that this is not automatic:
   On abelian surfaces, for instance,
   $\rho(X)$
   varies between 1 and 4, but one has always $z(X)=1$, since
   the intersection of the nef cone and the big cone is the
   only Zariski chamber.
   The same thing happens on suitable K3 surfaces: There are
   K3 surfaces $X$ of any Picard number up to 11 with $z(X)=1$
   (see \cite[Theorem~2]{Kov94}).
   On the other hand, if one considers the blow-up $X_r$
   of $\P^2$ in
   $r\ge 9$ general points, then the surface
   $X_r$ (which is no longer a Del Pezzo surface)
   contains infinitely
   many $(-1)$-curves and therefore one has $z(X_r)=\infty$.

   Our algorithm -- to be discussed in Sect.~\ref{sect-algo} --
   is in no way
   specific to Del Pezzo surfaces. It applies to any surface
   where the irreducible curves with negative self-intersection are
   explicitly known. We plan to study further applications of
   this method in a subsequent paper.

\paragraph*{\it Acknowledgement.} We
      benefited from discussions with V.~Welker.

\section{Negative curves and chambers}\label{sect-chambers}
\enlargethispage{\baselineskip}

   Consider a smooth projective surface $X$. A divisor $D$ on $X$
   is \emph{big}, if its \emph{volume}
   $$
      {\rm vol}_X(D)\stackrel{\rm def}{=}\limsup_k{\frac{h^0(X,kD)}{k^2/2}}
   $$
   is positive.
   The
   \emph{big cone} $\BigCone(X)$ is the cone in the N\'eron-Severi
   vector space $\NS_\R(X)$ that is generated by the big divisors.
   To any big and nef
   $\R$-divisor $P$, one associates the \textit{Zariski
   chamber} $\Sigma_P$, which by definition
   consists of all divisors in
   $\BigCone(X)$ such
   that the irreducible curves
   in the negative part of
   the Zariski decomposition of $D$ are precisely the curves $C$
   with
   $P\cdot C=0$.
   It is shown in \cite[Lemma~1.6]{BKS} that
   for any two big and nef divisors
   $P$ and $P'$, the Zariski
   chambers $\Sigma_P$ and $\Sigma_{P'}$ are
   either equal or disjoint. So
   the Zariski chambers yield a decomposition of the
   big cone.
   If $A$ is an ample divisor, then the chamber $\Sigma_A$
   is the intersection of the big cone and the nef cone,
   and its interior is the ample cone;
   in the sequel we call it the \textit{nef chamber} for short.
   The main
   result of \cite{BKS} states that
   the decomposition into Zariski chambers is
   a locally finite decomposition of $\BigCone(X)$
   into rational locally
   polyhedral subcones, such that
   \begin{itemize}\compact
   \item
      on each chamber the volume
      function is given by a single polynomial of degree two, and
   \item
      in the interior of each chamber the stable base
      loci are constant.
      (See Proposition~\ref{prop-stable-base-loci} below for
      the general statement.)
   \end{itemize}

   The following characterization will be essential for our
   purposes.

\begin{proposition}\label{prop-neg-def}
   The set of Zariski chambers on a smooth projective surface
   $X$ that are different from the
   nef chamber
   is in bijective correspondence with the set of
   reduced divisors on $X$ whose intersection matrix is negative
   definite.
\end{proposition}

\begin{proof}
   Given a chamber $\Sigma_P$, we consider the irreducible
   curves $\liste C1r$ with $P\cdot C_i=0$. Then the divisor
   $C_1+\dots+C_r$ has negative definite intersection matrix
   thanks to the index theorem.

   Conversely, given a reduced divisor
   $C_1+\dots+C_r$ with negative definite intersection matrix,
   we consider
   the divisor
   $$
      D \eqdef H+k(C_1+\dots+C_r) \ ,
   $$
   where $H$ is a fixed ample divisor and $k$ a positive integer.
   This divisor is big, and we claim that
   for $k\gg0$ the negative part of
   its Zariski decomposition will have
   $C_1\cup\dots\cup C_r$ as its support.
   The latter fact can for instance
   be seen from the computation of the
   Zariski decomposition according to \cite{Bau}.
   Alternatively, consider the linear system of equations
   \begin{equation}\label{linsys}
      (H+\sum_{i=1}^r a_iC_i)C_j=0\,, \qquad j=1,\dots,r,
   \end{equation}
   with unknowns $a_1,\dots,a_r$. If $S$ denotes the
   intersection matrix $(C_i\cdot C_j)_{i,j}$, then the unique
   solution of \eqnref{linsys} is given by
   $$
      \vect{a_1\\ \vdots \\ a_r}=-S\inverse\vect{H\cdot C_1\\
      \vdots\\ H\cdot C_r}
   $$
   As $S$ is by assumption negative definite, it follows that all
   entries of $S\inverse$ are $\le 0$ (see~\cite[Lemma~4.1]{BKS}),
   and consequently we have
   $a_i\ge 0$ for all $i$.
   The divisor $H+\sum_{i=1}^r a_iC_i$ is then for $k\gg 0$
   clearly an effective and
   nef $\mathbb Q$-subdivisor of $H+k\sum_{i=1}^r C_i$
   having zero intersection with all $C_i$. By the uniqueness of
   Zariski decompositions, it follows that it is the positive
   part in the Zariski decomposition of $H+k\sum_{i=1}^r
   C_i$, and therefore the negative part has support
   $C_1\cup\dots\cup C_r$, as claimed.
\end{proof}

\begin{remark}\rm
   Note that the divisor $D=H+k(C_1+\dots+C_r)$ that is
   considered in the proof of Proposition~\ref{prop-neg-def} lies
   in the \emph{interior} of the chamber that corresponds to
   $C_1+\dots+C_r$. In fact, write $D=P+N$ for its Zariski
   decomposition, and suppose that $D$ lies on the boundary of
   a chamber. Then by \cite[Proposition~1.7]{BKS} there must
   exist an irreducible curve $C\subset X$
   with $P\cdot C=0$
   that does not occur as
   a component of $N$.
   But as $P$ is of the form $H+a_1C_1+\dots+a_rC_r$ with $H$
   ample, it is clear that $P\cdot C=0$ can happen only if $C$ is
   among
   the curves $C_i$. However, all of them are components of $N$.
\end{remark}

   The next statement justifies the claim made in the
   introduction to the effect that counting Zariski chambers is
   equivalent to counting stable base loci of big linear series.
   By way of notation, we write $\B(|D|)$ for the base locus of the
   linear series~$|D|$, and
   $$
      \SB(D)\eqdef\bigcap_{m=1}^\infty\B(|mD|)
   $$
   for the stable base locus of $D$.

\begin{proposition}\label{prop-stable-base-loci}
   The set of Zariski chambers on a smooth projective surface
   $X$
   is in bijective correspondence with the set of
   stable base loci that occur in big linear series on $X$.
\end{proposition}

\begin{proof}
   As we already said above, it follows from \cite{BKS} that for a
   divisor $D$ that lies in the interior of a Zariski chamber, the stable base
   locus $\SB(D)$
   coincides with
   the support of the negative part of the Zariski
   decomposition of $D$.
   The point to show is therefore that the big
   divisors whose numerical classes lie on boundaries
   of Zariski chambers cannot
   lead to stable base loci that have not been accounted for by
   the
   divisors in the interior of chambers. To see that latter,
   suppose that
   $D$ is a big divisor on $X$.
   If $A$ is
   any ample $\mathbb Q$-divisor $A$, then
   we have
   \begin{equation}\label{sb-inclusion}
      \SB(D)\subset\SB(D-A) \,.
   \end{equation}
   For a suitable choice of $A$, the
   numerical class of the
   divisor $D-A$ does not lie on
   the boundary of any chamber. Moreover, as $D$ is big,
   $D-A$ is
   still big when $A$ is sufficiently small. As $D-A$ then lies
   in the interior of a Zariski chamber, $\SB(D-A)$ is the
   support of the negative part of a Zariski decomposition,
   and hence it is the support of a divisor
   $C_1+\dots+C_r$ with negative definite intersection matrix.
   But then $\SB(D)$ is by \eqnref{sb-inclusion}
   a subdivisor of this divisor, and hence
   has negative definite intersection matrix as well. By
   Proposition~\ref{prop-neg-def} this divisor corresponds to a
   Zariski chamber, and hence has been accounted for already.
\end{proof}

\begin{remark}\rm
   Note that in general the stable base locus $\SB(D)$ does not
   depend only on the numerical equivalence class of $D$ (see
   \cite[Example~10.3.3]{PAG}).
   In order to get a function on the big cone, one considers
   \emph{augmented} base loci
   instead (see \cite[Sect~10.3]{PAG}).
   In light of this fact it is even
   more surprising that by Proposition~\ref{prop-stable-base-loci}
   all stable base loci on surfaces are accounted for by
   the Zariski chambers. For instance, in the cited
   Example~\cite[10.3.3]{PAG} one has two numerically
   equivalent big and nef divisors
   $D_1$ and $D_2$ such that $\SB(D_1)=\emptyset$ and $\SB(D_2)$
   is a curve.
   According to Proposition~\ref{prop-stable-base-loci}
   these stable base loci correspond to two
   distinct Zariski chambers.
\end{remark}

   Our aim now is to study the number $z(X)$ of Zariski
   chambers on $X$.
   By way of terminology, the term
   \emph{negative curve} will always mean an
   irreducible curve with negative self-intersection.
   Two things about $z(X)$ are clear from the outset:
   \begin{itemize}\compact
   \item[(1)]
      If $X$ carries only a finite number $N$ of negative curves,
      then one has the trivial upper bound
      $$
         z(X)\le 2^N \,.
      $$
      Intuitively, it seems unlikely that $z(X)$
      is equal (or close) to this upper bound,
      as this would mean that
      every (or almost every)
      set of negative curves occurs in a
      stable base locus.
   \item[(2)]
      We have $z(X)=\infty$ if and only if there are infinitely many
      negative curves on~$X$.
      The blow-up of $\P^2$ in $\ge 9$ general points gives such
      an example.
   \end{itemize}

   When the negative curves on $X$ are known explicitly,
   then there is a way
   to effectively determine the number $z(X)$.
   To formulate the enumerative
   statement, we will use for a given $(n\times n)$-matrix
   the notion
   \textit{principal submatrix} to mean as usual
   a submatrix that
   arises by deleting $k$ corresponding rows and columns of the
   matrix, where $0\le k<n$.
   The following is then an immediate
   consequence of Proposition~\ref{prop-neg-def}:

\begin{proposition}\label{prop-number-of-chambers}
   Let $X$ be a smooth projective surface
   that contains only
   finitely many negative curves.
   \begin{itemize}\compact
   \item[(i)]
   We have
   $$
      z(X)=
      1+\#\sset{\tab[c]{
      negative definite principal submatrices\\ of the intersection
      matrix of the negative curves on $X$}}\,.
   $$
   \item[(ii)]
   More generally,
   let $\liste C1r$ be distinct negative curves on $X$,
   and let $S$
   be their intersection matrix. Then the number of Zariski
   chambers that
   are supported by a non-empty
   subset of $\set{\liste C1r}$ equals the
   number of negative definite principal submatrices of the matrix
   $S$.
   \end{itemize}
\end{proposition}
   Strictly speaking, it is of course
   not actually the submatrices themselves
   that are to be counted, but the subsets
   of the index set $\set{1,\dots,r}$
   that give rise to the submatrices.
   Nonetheless, we will generally use
   this shorter formulation in the sequel.
   Also,
   note that the \engqq{1+} in (i) accounts for the
   nef chamber.

\begin{remark}\rm
   Looking at Proposition~\ref{prop-number-of-chambers}, one
   would wish for a general matrix-theoretic result that gives
   information about the number of negative definite principal
   submatrices in terms of other (easier accessible) quantities
   associated with the matrix. It seems however that no results
   in this direction are available so far. Not even is it clear
   which quantities might be of relevance: The probably most naive
   guess might be to consider the signature $(p,n)$ of the
   matrix,
   where $p$ is the number of positive and $n$ the number of
   negative eigenvalues.
   However,
   as the following two examples show,
   one cannot expect useful bounds
   in terms
   of the signature.

   (i) Consider the matrix $A$ that is diagonally composed of a
   $k\times k$ unit matrix and the negative of an
   $\ell\times\ell$ unit matrix. Its signature is $(p,n)=(k,\ell)$, and
   it has exactly $2^k-1$ positive definite principal
   submatrices.

   (ii) On the other hand, take $A$ to be
   diagonally composed of a $k\times k$ unit matrix and $\ell$
   copies of the matrix
   $$
      \matr{0 & -1 \\ -1 & 0} \ .
   $$
   It has the same number
   $2^k-1$ of positive definite principal submatrices, but its
   signature is $(p,n)=(k+\ell, \ell)$.

   So while in (i)
   the number of positive definite
   principal submatrices depends only on $p$,
   it depends in (ii)
   on the difference $p-n$.
\end{remark}

\section{Computing chambers}\label{sect-algo}

   Proposition~\ref{prop-number-of-chambers} suggests a way to
   effectively determine Zariski chambers when the
   numerical classes of the
   negative curves are explicitly known: Each negative definite
   principal submatrix of the intersection matrix of the negative
   curves corresponds to a chamber, supported by the
   curves that are represented by the chosen rows and columns.
   Determining the negative
   definite submatrices is however in practice not
   at all immediate:
   If there are many negative curves, then
   such work cannot be done by hand.
   And even when carried out by computer, it is
   not a viable course of action to apply brute
   force and check
   \emph{all} submatrices for
   negative definiteness:
   For instance, on
   the Del Pezzo surface $X_8$ there
   are $2^{240}$ potential submatrices.
   Our algorithm
   exploits the following two
   observations, which drastically reduce the complexity of the
   computation:
   \begin{itemize}
   \item[(1)]
      Let $A$ be the intersection matrix of $n$ negative curves.
      If the principal submatrix $A_S$ corresponding to a subset
      $S\subset\set{1,\dots,n}$ is not negative definite, then
      none of the subsets $S'$ with $S'\supset S$ need to be
      examined, since they cannot be negative definite.
      One can therefore use a backtracking strategy.
   \item[(2)]
      Let $S$ be a subset and let $T$ be the set obtained from
      $S$ by removing its largest element. If the subsets are
      treated in such an order that $S$ is only examined after
      $A_T$ has turned out to be negative definite, then the
      negative definiteness of $A_S$ can be read off the sign of
      its determinant.
   \end{itemize}

   The algorithm below generates all \emph{positive
   definite} principal submatrices of a given symmetric matrix.
   It will subsequently be applied to the negative of the intersection
   matrix.

\begin{algorithm}\label{algo}\rm
   The algorithm takes as input an integer $n\ge 1$ and a
   symmetric
   $(n\times n)$-matrix $A$ over $\R$. It outputs all subsets
   $S\subset\set{1,\dots,n}$ having the property that
   the corresponding principal submatrix
   $A_S$ is
   positive definite.

\begin{lines}
   input $n$, $A$
   $k \gets 1$
   $S \gets \set{1}$
   while $S\ne\emptyset$ do
   \+ assert($k=\max S$ and $A_{S\setminus\set k}$ is positive definite)
   \+ if $\det A_S>0$ then
   \+\+ output $S$
   \+ else
   \+\+ $S \gets S \setminus\set{k}$
   \+ end if
   \+ assert($k\ge\max S$ and $A_S$ is positive definite)
   \+ if $k < n$ then
   \+\+ $k \gets k + 1$
   \+\+ $S \gets S \union\set{k}$
   \+ else
   \+\+ $S \gets S\setminus\set k$
   \+\+ if $S\ne\emptyset$ then
   \+\+\+ $k \gets \max S$
   \+\+\+ $S \gets S \setminus\set k$
   \+\+\+ $k \gets k + 1$
   \+\+\+ $S \gets S \union\set k$
   \+\+ end if
   \+ end if
   end while
\end{lines}
\end{algorithm}

\begin{remark}\rm
   The gain in efficiency compared to checking \emph{all}
   principal submatrices is considerable -- and in fact crucial
   for the algorithm to be practical at all.
   For instance, on the Del Pezzo surface $X_6$ the algorithm
   checks only $15600$ submatrices
   instead of all $2^{27}=134217728$ submatrices,
   which means
   reducing cases to about $0.01$ percent.
\end{remark}

\begin{proof}[Proof of correctness and termination]
   Note first that the two assertions made within the loop are
   true whenever the algorithm reaches them
   (the empty matrix being considered
   positive definite).
   Therefore the condition that $A_S$ be positive definite
   is equivalent to $\det A_S>0$.
   We now have to show that the algorithm terminates and that it
   outputs precisely the claimed subsets.
   Readers familiar with backtracking algorithms might rather
   quickly understand the strategy of Algorithm~\ref{algo} and
   can argue from there.
   For
   readers not versed in these matters we will provide an
   explicit alternative view as follows.

   For index sets $S, S'\subset\set{1,\dots,n}$ we write $S<S'$
   if for some integer $\ell\ge 0$ we have
   $$
      S\intersect\set{1,\dots,\ell}=
      S'\intersect\set{1,\dots,\ell}
   $$
   and
   $$
      \min(S\setminus\set{1,\dots,\ell}) <
      \min(S'\setminus\set{1,\dots,\ell}) \ ,
   $$
   where we set $\min(\emptyset)=-\infty$.
   It is immediate that \engqq{$<$} is a strict total
   order on the set of subsets of $\set{1,\dots,n}$.
   Correctness and termination follow then from the two following
   claims.
   \begin{itemize}\compact
   \item[(i)]
      Loop invariant: At the beginning and at the
      end of each loop cycle all
      index sets $T<S$ have been output for which $A_T$
      is positive definite.
   \item[(ii)]
      At the end of each loop cycle either
      the value of $S$ is strictly bigger
      than at the beginning, or $S=\emptyset$
      (in which case it is the last cycle).
   \end{itemize}

   To verify this, let $S_1$ and $S_2$ be the values of the variable
   $S$ at the beginning and at the end of a loop cycle respectively, and write
   $S_1=\set{\liste i1m}$ with $i_1<\dots<i_m$.
   Then we have
   \begin{equation}\label{S1S2}
      \begin{array}{@{}ll}
         S_2=\set{\liste i1m,i_m+1} & \mbox{if $i_m<n$ and $A_{S_1}$ is positive definite,} \\
         S_2=\set{\liste i1{m-1},i_m+1} & \mbox{if $i_m<n$ and $A_{S_1}$ is not positive definite,} \\
         S_2=\set{\liste i1{m-2},i_{m-1}+1} \mbox{ or } S_2=\emptyset & \mbox{if $i_m=n$.}
      \end{array}
   \end{equation}
   So we have $S_2>S_1$ or $S_2=\emptyset$ in each case, which
   proves
   Claim (ii). As for Claim~(i): The algorithm
   clearly outputs $S_1$, if $A_{S_1}$ is positive definite.
   Further, one sees from \eqnref{S1S2} that there is
   no set $T$ with $S_1<T<S_2$ unless $i_m<n$ and $A_{S_1}$ is
   not positive definite. In the latter case, all sets $T$ with
   $S_1<T<S_2$ are supersets of $S_1$, and hence none of the
   corresponding matrices $A_T$ can be
   positive definite.
\end{proof}

\section{Del Pezzo surfaces}

   Our aim is now to apply Algorithm~\ref{algo} to the Del Pezzo
   surfaces $X_r$ for $1\le r\le 8$, which are
   the blow-ups of $\P^2$ at
   $r$ general points.
   To this end, we first need to describe
   all negative curves on the surfaces
   $X_r$. They have been determined by Manin:

\begin{theorem}[Manin {\cite[Chapt.~IV]{Man}}]
   The negative curves on $X_r$ are
   \begin{itemize}
   \item[(1)]
      the exceptional divisors corresponding to the blown-up points $\liste
      p1r$
   \end{itemize}
   and the proper transforms of the following curves in $\P^2$:
   \begin{itemize}\compact
   \item[(2)]
      the lines through pairs of points $p_i, p_j$;
   \item[(3)]
      if $r\ge 5$, the conics through 5 points from $\liste p1r$;
   \item[(4)]
      if $r\ge 7$, the cubics through 7 points from $\liste p1r$ with a
      double point in one of them;
   \item[(5)]
      if $r=8$, the quartics through the 8 points $\liste p18$ with
      double points in 3 of them;
   \item[(6)]
      if $r=8$, the quintics through the 8 points $\liste p18$ with
      double points in 6 of them;
   \item[(7)]
      if $r=8$, the sextics through the 8 points $\liste p18$
      with double points in 7 of them and a triple point in one
      of them.
   \end{itemize}
\end{theorem}

   The proof in \cite{Man} works from the more general
   perspective of root systems. We believe that it can also be
   useful to have a very quick argument for this basic
   result in the spirit of \cite[Theorem~V.4.9]{Har77}, and we
   provide such an argument below. Since we will at any rate need to
   describe the classes of the
   negative curves and their intersection behaviour
   for our purposes, doing so means little additional effort.

\begin{proof}
   (i) We start by showing that negative curves as asserted in
   (2) to (7) exist. An immediate dimension count shows that on
   $\P^2$ there are in any event effective \emph{divisors}
   (which may be reducible) having \emph{at least} the indicated
   multiplicities.
   Writing $H=\pi^*\O_{\P^2}(1)$, $E_i=\pi\inverse(p_i)$, and $E=\longsum E1r$,
   these divisors on $\P^2$
   correspond to effective divisors in the
   following
   linear series on $X_r$:
   \begin{equation}\label{eqn-negative-curves}
      \renewcommand\arraystretch{1.4}
      \arraycolsep=0.139em
      \begin{array}{rcl@{\qquad}l}
      \CC1_{ij}&=&H-E_i-E_j              & 1\le i<j\le r \\
      \CC2 &=& 2H-E & (\mbox{if } r=5) \\
      \CC2_i &=& 2H-E+E_i & 1\le i\le 6 \quad (\mbox{if } r=6) \\
      \CC2_{ij} &=& 2H-E+E_i+E_j & 1\le j<j \le 7 \quad (\mbox{if } r=7) \\
      \CC2_{ijk}&=&2H-E+E_i+E_j+E_k     & 1\le i<j<k \le 8 \quad (\mbox{if } r=8) \\
      \CC3_i &=& 3H-E-E_i & 1\le i\le 7 \quad (\mbox{if } r=7) \\
      \CC3_{ij}&=&3H-E-E_i+E_j           & 1\le i, j\le 8,\ i\ne j \quad (\mbox{if } r=8) \\
      \CC4_{ijk}&=&4H-E-E_i-E_j-E_k     & 1\le i<j<k \le 8 \quad (\mbox{if } r=8) \\
      \CC5_{ij}&=&5H-2E+E_i+E_j          & 1\le i<j\le 8 \quad (\mbox{if } r=8)  \\
      \CC6_{i}&=&6H-2E-E_i               & 1\le i\le 8 \quad (\mbox{if } r=8)
      \end{array}
   \end{equation}
   The point is to show
   that these divisors are irreducible.
   To see this, one
   checks first that if $C$ is any of these divisors, then
   one
   has
   \begin{equation}\label{eqn-C}
      C^2=-1 \quad\mbox{and}\quad -K_{X_r}\cdot C=1 \ .
   \end{equation}
   As $-K_{X_r}$ is ample, the second equation implies then that
   $C$ must be irreducible. In particular, its image curve
   on $\P^2$
   has exactly the asserted multiplicities.

   (ii) It remains to show that the curves in (1) to (7) are the
   only negative curves on $X_r$. So suppose that $C\subset X_r$
   is any negative curve that is
   different from the exceptional curves of
   the blow-up.
   Via the adjunction formula it follows from the ampleness of
   $-K_{X_r}$ that the equations \eqnref{eqn-C} hold for $C$.
   One has
   $C\in|dH-\sum_{i=1}^r m_iE_i|$ for suitable integers $d\ge 1$ and
   $m_i\ge 0$.
   We claim that
   \begin{equation}\label{d-bounds}
      \begin{array}{l@{\qquad\mbox{if }}l}
      d\le 2 & r\le 6, \\
      d\le 3 & r= 7, \\
      d\le 6 & r= 8.
      \end{array}
   \end{equation}
   To prove \eqnref{d-bounds}, note first that the equations
   \eqnref{eqn-C} translate to
   \begin{equation}\label{eqn-translated}
      d^2-\sum m_i^2=-1 \quad\mbox{and}\quad 3d-\sum m_i=1 \ .
   \end{equation}
   Upon combining these equations with the
   Cauchy-Schwarz inequality
   $$
      \Big(\sum_{i=1}^r m_i\Big)^2\le r\sum_{i=1}^r m_i^2 \ ,
   $$
   we get a quadratic equation for $d$, which in turn implies
   $d\le 2$ for $r\le 6$, as well as $d\le 3$ for $r=7$
   and $d\le 7$ for $r=8$.
   So the claim \eqnref{d-bounds} will be established as soon
   as we can rule out the possibility that
   $d=7$ and $r=8$. In that case
   we would have equality in the Cauchy-Schwarz inequality, and
   therefore
   $m_1=\dots=m_8$. But then \eqnref{eqn-translated} would imply
   $m_i=5/2$, which is impossible.

   To complete the proof, one checks now that the equations
   \eqnref{eqn-translated} have only the solutions corresponding
   to
   the classes in \eqnref{eqn-negative-curves}.
   This can
   be
   done by trial, since
   the bounds \eqnref{d-bounds} on $d$
   leave only finitely many
   possibilities for the integers $m_i$.
\end{proof}

   One sees from \eqnref{eqn-negative-curves} that
   the number $N$ of negative curves on $X_r$ is given by the following table:
   \begin{equation}\label{eqn-number-of-curves}
      \begin{array}{c|*8c}
      r &  1 & 2 & 3 & 4 & 5 & 6 &  7&  8 \tabline
      N &  1 & 3 & 6 &10 &16 &27 &56 &240
      \end{array}
   \end{equation}

\begin{figure}
   \begin{center}
   \begin{scriptsize}
      \arraycolsep=0.08em
      $\left(\begin{array}{*{27}{r}}
         -1& 0& 1& 0& 1& 0& 0& 1& 0& 0& 0& 1& 0& 0& 0& 1& 0& 1& 0& 0& 0& 0& 0& 1& 1& 1& 1 \\
         0 &-1& 1& 0& 0& 1& 0& 0& 1& 0& 0& 0& 1& 0& 0& 1& 0& 0& 1& 0& 0& 0& 1& 0& 1& 1& 1 \\
         1 &1 &-1& 0& 0& 0& 0& 0& 0& 1& 0& 0& 0& 1& 1& 0& 0& 0& 0& 1& 1& 1& 1& 1& 0& 0& 0 \\
         0 &0 &0 &-1& 1& 1& 0& 0& 0& 1& 0& 0& 0& 1& 0& 1& 0& 0& 0& 1& 0& 0& 1& 1& 0& 1& 1 \\
         1 &0 &0 &1 &-1& 0& 0& 0& 1& 0& 0& 0& 1& 0& 1& 0& 0& 0& 1& 0& 1& 1& 1& 0& 1& 0& 0 \\
         0 &1 &0 &1 &0 &-1& 0& 1& 0& 0& 0& 1& 0& 0& 1& 0& 0& 1& 0& 0& 1& 1& 0& 1& 1& 0& 0 \\
         0 &0 &0 &0 &0 &0 &-1& 1& 1& 1& 0& 0& 0& 0& 1& 1& 0& 0& 0& 0& 1& 0& 1& 1& 1& 0& 1 \\
         1 &0 &0 &0 &0 &1 &1 &-1& 0& 0& 0& 0& 1& 1& 0& 0& 0& 0& 1& 1& 0& 1& 1& 0& 0& 1& 0 \\
         0 &1 &0 &0 &1 &0 &1 &0 &-1& 0& 0& 1& 0& 1& 0& 0& 0& 1& 0& 1& 0& 1& 0& 1& 0& 1& 0 \\
         0 &0 &1 &1 &0 &0 &1 &0 &0 &-1& 0& 1& 1& 0& 0& 0& 0& 1& 1& 0& 0& 1& 0& 0& 1& 1& 0 \\
         0 &0 &0 &0 &0 &0 &0 &0 &0 &0 &-1& 1& 1& 1& 1& 1& 0& 0& 0& 0& 0& 1& 1& 1& 1& 1& 0 \\
         1 &0 &0 &0 &0 &1 &0 &0 &1 &1 &1 &-1& 0& 0& 0& 0& 0& 0& 1& 1& 1& 0& 1& 0& 0& 0& 1 \\
         0 &1 &0 &0 &1 &0 &0 &1 &0 &1 &1 &0 &-1& 0& 0& 0& 0& 1& 0& 1& 1& 0& 0& 1& 0& 0& 1 \\
         0 &0 &1 &1 &0 &0 &0 &1 &1 &0 &1 &0 &0 &-1& 0& 0& 0& 1& 1& 0& 1& 0& 0& 0& 1& 0& 1 \\
         0 &0 &1 &0 &1 &1 &1 &0 &0 &0 &1 &0 &0 &0 &-1& 0& 0& 1& 1& 1& 0& 0& 0& 0& 0& 1& 1 \\
         1 &1 &0 &1 &0 &0 &1 &0 &0 &0 &1 &0 &0 &0 &0 &-1& 0& 1& 1& 1& 1& 1& 0& 0& 0& 0& 0 \\
         0 &0 &0 &0 &0 &0 &0 &0 &0 &0 &0 &0 &0 &0 &0 &0 &-1& 1& 1& 1& 1& 1& 1& 1& 1& 1& 1 \\
         1 &0 &0 &0 &0 &1 &0 &0 &1 &1 &0 &0 &1 &1 &1 &1 &1 &-1& 0& 0& 0& 0& 1& 0& 0& 0& 0 \\
         0 &1 &0 &0 &1 &0 &0 &1 &0 &1 &0 &1 &0 &1 &1 &1 &1 &0 &-1& 0& 0& 0& 0& 1& 0& 0& 0 \\
         0 &0 &1 &1 &0 &0 &0 &1 &1 &0 &0 &1 &1 &0 &1 &1 &1 &0 &0 &-1& 0& 0& 0& 0& 1& 0& 0 \\
         0 &0 &1 &0 &1 &1 &1 &0 &0 &0 &0 &1 &1 &1 &0 &1 &1 &0 &0 &0 &-1& 0& 0& 0& 0& 1& 0 \\
         0 &0 &1 &0 &1 &1 &0 &1 &1 &1 &1 &0 &0 &0 &0 &1 &1 &0 &0 &0 &0 &-1& 0& 0& 0& 0& 1 \\
         0 &1 &1 &1 &1 &0 &1 &1 &0 &0 &1 &1 &0 &0 &0 &0 &1 &1 &0 &0 &0 &0 &-1& 0& 0& 0& 0 \\
         1 &0 &1 &1 &0 &1 &1 &0 &1 &0 &1 &0 &1 &0 &0 &0 &1 &0 &1 &0 &0 &0 &0 &-1& 0& 0& 0 \\
         1 &1 &0 &0 &1 &1 &1 &0 &0 &1 &1 &0 &0 &1 &0 &0 &1 &0 &0 &1 &0 &0 &0 &0 &-1& 0& 0 \\
         1 &1 &0 &1 &0 &0 &0 &1 &1 &1 &1 &0 &0 &0 &1 &0 &1 &0 &0 &0 &1 &0 &0 &0 &0 &-1& 0 \\
         1 &1 &0 &1 &0 &0 &1 &0 &0 &0 &0 &1 &1 &1 &1 &0 &1 &0 &0 &0 &0 &1 &0 &0 &0 &0 &-1
      \end{array}\right)$

   \end{scriptsize}
   \caption{\label{figure-27-lines}
      The intersection matrix $A_{6}$ of the 27 lines on a smooth
      cubic surface, obtained as a submatrix of $A_8$
      as described in
      Sect.~\ref{sect-proof}.}
   \end{center}
\end{figure}

\section{Proof of the theorem}\label{sect-proof}

   We now turn to the proof of the theorem stated in the
   introduction. We start by determining the intersection
   products of the negative curves on $X_r$.
   Note that
   it is enough to write down the intersection matrix $A_8$
   of the negative curves on the surface $X_8$:
   The intersection matrices $A_r$ for the surfaces $X_r$, $r<8$,
   can then be obtained by taking the principal submatrices
   corresponding to those curves whose classes are contained in
   $\Z\cdot[H]\oplus\bigoplus_{i=1}^r[E_i]$.

   In order to get a
   compact statement that is suitable for computations, we will
   use for tuples of integers $(\liste i1m)$ and $(\liste j1n)$
   the abbreviation
   $$
      (\liste i1m)*(\liste j1n)
         =\sum_{{\mu=1,\dots,m} \atop
         {\nu=1,\dots,n}}\mbox{sign}(i_\mu)\cdot\mbox{sign}(i_\nu)\cdot\dd{\left|i_\mu\right|}{\left|j_\nu\right|}
         \ ,
   $$
   where $\delta$ is the Kronecker delta.
   Keeping the notation for the curves on $X_8$
   and the index ranges as in~\eqnref{eqn-negative-curves},
   we find:

   $$ \renewcommand\arraystretch{1.3}
      \begin{array}{rcl@{\quad}rcl}
         E_i\cdot        E_\ell             &=& (-i)*(\ell)            &            \CC2_{ijk}\cdot \CC3_{\ell m}      &=& 1 + (i,j,k)*(\ell,-m)  \\
         E_i\cdot        \CC1_{\ell m}      &=& (i)*(\ell,m)           &            \CC2_{ijk}\cdot \CC4_{\ell m n}    &=& (i,j,k)*(\ell, m, n) \\
         E_i\cdot        \CC2_{\ell m n}    &=& 1 - (i)*(\ell,m,n)     &            \CC2_{ijk}\cdot \CC5_{\ell m}      &=& 2 - (i,j,k)*(\ell,m) \\
         E_i\cdot        \CC3_{\ell m}      &=& 1 + (i)*(\ell,-m)      &            \CC2_{ijk}\cdot \CC6_{\ell}        &=& 1 + (i,j,k)*(l) \\
         E_i\cdot        \CC4_{\ell m n}    &=& 1 + (i)*(\ell,m,n)     &            \CC3_{ij}\cdot  \CC3_{\ell m}      &=& 1 + (-i,j)*(\ell,-m) \\
         E_i\cdot        \CC5_{\ell m}      &=& 2 - (i)*(\ell,m)       &            \CC3_{ij}\cdot  \CC4_{\ell m n}    &=& 1 + (-i,j)*(\ell,m,n) \\
         E_i\cdot        \CC6_{\ell}        &=& 2 + (i)*(\ell)         &            \CC3_{ij}\cdot  \CC5_{\ell m}      &=& 1 + (i,-j)*(\ell,m) \\
         \CC1_{ij}\cdot  \CC1_{\ell m}      &=& 1 - (i,j)*(\ell,m)     &            \CC3_{ij}\cdot  \CC6_{\ell}        &=& 1 + (-i,j)*(\ell) \\
         \CC1_{ij}\cdot  \CC2_{\ell m n}    &=& (i,j)*(\ell,m,n)       &            \CC4_{ijk}\cdot \CC4_{\ell m n}    &=& 2 - (i,j,k)*(\ell,m,n) \\
         \CC1_{ij}\cdot  \CC3_{\ell m}      &=& 1 + (i,j)*(-\ell,m)    &            \CC4_{ijk}\cdot \CC5_{\ell m}      &=& (i,j,k)*(\ell,m) \\
         \CC1_{ij}\cdot  \CC4_{\ell m n}    &=& 2 - (i,j)*(\ell,m,n)   &            \CC4_{ijk}\cdot \CC6_{\ell}        &=& 1 - (i,j,k)*(\ell) \\
         \CC1_{ij}\cdot  \CC5_{\ell m}      &=& 1 + (i,j)*(\ell,m)     &            \CC5_{ij}\cdot  \CC5_{\ell m}      &=& 1 - (i,j)*(\ell,m) \\
         \CC1_{ij}\cdot  \CC6_{\ell}        &=& 2 - (i,j)*(\ell)       &            \CC5_{ij}\cdot  \CC6_{\ell}        &=& (i,j)*(\ell) \\
         \CC2_{ijk}\cdot \CC2_{\ell m n}    &=& 2 - (i,j,k)*(\ell,m,n) &            \CC6_i\cdot     \CC6_{\ell}        &=& (-i)*(\ell)
      \end{array}
   $$
   The preceding formulas determine
   the intersection matrix $A_8$,
   which is of dimension 240.
   As described above, the matrices $A_r$
   for $r=1,\dots,7$
   are obtained as submatrices thereof.
   They are
   of dimension
   $1$,
   $3$,
   $6$,
   $10$,
   $16$,
   $27$, and
   $56$
   respectively (see \eqnref{eqn-number-of-curves}).
   As an example, we display the matrix $A_6$ in
   Figure~\ref{figure-27-lines}.
   Using Algorithm~\ref{algo}, applied to the matrix $-A_r$, we
   obtain the
   number of negative definite principal submatrices of $A_r$:
   $$
      \begin{array}{c|*8c}
         r        & 1 & 2 & 3 & 4 & 5 & 6 & 7 & 8 \tabline
         \#       & 1 & 4 & 17 & 75 & 392 & 2\,763 & 33\,644 & 1\,501\,680
      \end{array}
   $$
   Proposition~\ref{prop-number-of-chambers} then gives
   part (i) of the theorem. With an obvious
   modification of Algorithm~\ref{algo} we obtain in each
   case also the
   maximal cardinality of the positive definite index sets, which
   shows that for each $r$ there are positive definite
   principal submatrices of $-A_r$ of dimension $r$. This proves
   part (ii) of the theorem.

\bigskip
\bigskip
\bigskip
\bigskip

\small
   Tho\-mas Bau\-er,
   Fach\-be\-reich Ma\-the\-ma\-tik und In\-for\-ma\-tik,
   Philipps-Uni\-ver\-si\-t\"at Mar\-burg,
   Hans-Meer\-wein-Stra{\ss}e,
   D-35032~Mar\-burg, Germany.

\nopagebreak
   \textit{E-mail address:} \texttt{tbauer@mathematik.uni-marburg.de}

\bigskip
   Michael Funke,
   Fach\-be\-reich Ma\-the\-ma\-tik und In\-for\-ma\-tik,
   Philipps-Uni\-ver\-si\-t\"at Mar\-burg,
   Hans-Meer\-wein-Stra{\ss}e,
   D-35032~Mar\-burg, Germany.

\nopagebreak
   \textit{E-mail address:} \texttt{funke@mathematik.uni-marburg.de}

\bigskip
   Sebastian Neumann,
   Albert-Ludwigs-Universit\"at Freiburg,
   Mathematisches Institut,
   Eckerstraße 1,
   D-79104 Freiburg, Germany

\nopagebreak
   \textit{E-mail address:} \texttt{sebastian.neumann@math.uni-freiburg.de}


\begin{thebibliography}{9}\footnotesize\itemsep=0cm\parskip=0cm

\bibitem{Bau}
   Bauer, Th.:
   A simple proof for the existence of Zariski decompositions on surfaces.
   J. Algebraic Geom. 18, 789-793 (2009)

\bibitem{BKS}
   Bauer, Th., Küronya, A., Szemberg, T.:
   Zariski chambers, volumes, and stable base loci.
   J. reine angew. Math. 576, 209-233 (2004)

\bibitem{AILB}
   Ein, L., Lazarsfeld, R., Mustata, M., Nakamaye, M., Popa, M.:
   Asymptotic invariants of line bundles. (English)
   Pure Appl. Math. Q. 1, No. 2, 379-403 (2005)

\bibitem{AIBL}
   Ein, L., Lazarsfeld, R., Mustata, M., Nakamaye, M., Popa, M.:
   Asymptotic invariants of base loci.
   Ann. Inst. Fourier 56, No. 6, 1701-1734 (2006)

\bibitem{Har77}
   Hartshorne, R.:
   Algebraic geometry.
   Graduate Texts in Mathematics 52, Springer-Verlag, New York Berlin Heidelberg, 1977

\bibitem{Kov94}
   Kov\'acs, S.J.:
   The cone of curves of a K3 surface.
   Math. Ann. 300, 681-691 (1994)

\bibitem{PAG}
   Lazarsfeld, R.:
   Positivity in Algebraic Geometry II. Springer-Verlag, 2004.

\bibitem{Man}
   Manin, Y.I.: Cubic Forms. Algebra, Geometry, Arithmetic. North-Holland Mathematical Library.
   Vol. 4. North-Holland, 1974.

\end{thebibliography}
\end{document}